%% file: paper_v2.3_FINAL_SIMODS.tex
\documentclass[twoside, 11pt]{article}

\usepackage{graphicx, subfigure}
\graphicspath{{figures/}}
\usepackage[margin=1in, footnotesep=0.5in]{geometry}
\usepackage[sort&compress, numbers]{natbib} 
\usepackage{amsfonts, amsmath, amssymb, amsthm}
\usepackage{MnSymbol}
\usepackage{mathtools}
\usepackage{enumitem}

\usepackage{color}
\definecolor{darkred}{RGB}{100,0,0}
\definecolor{darkgreen}{RGB}{0,100,0}
\definecolor{darkblue}{RGB}{0,0,150}

\usepackage{hyperref}
\hypersetup{colorlinks=true, linkcolor=darkred, citecolor=darkgreen, urlcolor=darkblue}
\usepackage{url}

\input notation.tex

\definecolor{purple}{rgb}{0.4,.1,.9}

\begin{document}
\thispagestyle{empty}

\title{Theoretical Foundations of Ordinal Multidimensional Scaling, Including Internal Unfolding and External Unfolding}
%\footnotetext{
%We would like to thank XXX.
%This work was partially supported by the US National Science Foundation XXX.}
\author{
Ery Arias-Castro\,\footnote{Department of Mathematics and Halıcıoğlu Data Science Institute, University of California, San Diego} 
\and
Clément Berenfeld\,\footnote{PreMeDICaL, INRIA, Universit\'e de Montpellier, France}
\and Daniel Kane\,\footnote{Department of Computer Science and Department of Mathematics, University of California, San Diego} 
}
\date{}
\maketitle

\begin{abstract}
We provide a comprehensive theory of multiple variants of ordinal multidimensional scaling, including internal unfolding and external unfolding. We first follow \citet{shepard1966metric} and work in a continuum model to gain insight. We then follow \citet{klein} and work in an asymptotic discrete setting.  

\end{abstract}

\section{Introduction} 
\label{sec:introduction}

{\em Multidimensional scaling (MDS)} is an umbrella name for various tasks and accompanying methods that aim at embedding a set of abstract items based on pairwise dissimilarity information. While MDS developed largely within Psychometrics, with early works dating back to the 1930s \cite{young1938discussion}, it is nowadays an integral part of multivariate analysis in Statistics \cite{MVA, seber2009multivariate} and of unsupervised learning in Machine Learning. The problem of MDS has been considered in other areas where it is known under different names, e.g., {\em embedding of metric spaces} in Mathematics and Computer Science \cite{blumenthal1938distance}; {\em Euclidean distance matrix completion} in Optimization \cite{laurent2001matrix}; and {\em sensor network localization} in Engineering \cite{priyantha2003anchor}. 
We provide further references throughout the article.

In its ordinal form, the basic problem of MDS consists in finding a configuration of points in a given Euclidean space whose pairwise distances agree in ranking as much as possible with a given (partial) ranking of the pairwise dissimilarities between the items.
This variant of the problem is particularly important in Psychometrics, where a human subject may be asked to compare objects in triads \cite{torgerson1952multidimensional} by answering questions such as ``Is item A closer to item B or item C?''.

Other popular variants of the problem are {\em internal unfolding} and {\em external unfolding}, as they are called in the Psychometrics literature. Internal unfolding is the problem of positioning individuals and objects in space based on preference data \cite{coombs1950psychological, bennett1956determination, tucker1960intra}. In the ordinal variant of the problem, a ranking of the objects is available for each individual. This is a special case of MDS where some of the dissimilarities --- those between individuals and those between objects --- are simply missing. 
In external unfolding \cite{gower1968adding, carroll1972individual, carroll1967}, the positions of the objects are known, i.e., the objects are already embedded. The problem is known as {\em lateration} in the Engineering literature \cite{glossary1994}. 
It is referred to as {\em interpolation} in \cite{delicado2024multidimensional}.

\subsection{Contribution and Content}
While most of the literature is dedicated to methods, there is much less available in terms of theory. In the present paper, we aim at giving a comprehensive theory of ordinal embedding in all these variants. More specifically, for each variant, we consider a {\em realizable} setting in which the items (both individuals and objects) are indeed points in a given Euclidean space and the ordinal information is congruent with the corresponding Euclidean distances. In that context, 
\begin{quote}
{\em We consider the fundamental question of whether there is enough information in the limit to recover the unknown positions up to the invariance inherent to the problem.}
\end{quote}
To gain insight, we first proceed as \citet{shepard1966metric} in his pioneering study of ordinal multidimensional scaling and pass to the limit to consider a continuum model. We then leverage the understanding gained from studying the embedding problem in the continuum model to better understand a discrete model inspired by the much more recent work of \citet{klein}.  

In \secref{external}, we consider ordinal external unfolding. 
In \secref{mds}, we consider ordinal multidimensional scaling in a setting where all triadic comparisons are available. We outline the reasoning of \citet{shepard1966metric} and contrast that with a different approach based on recent work by \citet{klein} and by ourselves \cite{arias2017some}. 
In \secref{internal}, we consider ordinal internal unfolding.
Each of these sections is structured as follows: We start by defining the problem in its commonly encountered discrete form; we then define, as \citet{shepard1966metric} did, a setting in the continuum that can be realistically seen as the limit of the discrete setting; and finally, we come back to the discrete setting where, by leveraging the insights gained from studying the problem in its continuum form, we establish a uniqueness result in the large-problem limit \`a la \citet{klein}. 
We close the paper with a brief discussion in \secref{discussion}.

We focus throughout on the {\em point model} in which the individuals are to be embedded as points \cite{coombs1950psychological}.
Note that we do not consider the sort of item response models featured in \cite{reckase2009multidimensional}, which have evolved from the ranking models that originated from early work of \citet{thurstone1927law}, \citet{bradley1952rank}, and others. 

The theory that we develop for external and internal unfolding is some of the only theory that we are aware of for these problems. While internal unfolding is known to be difficult in practice as discussed, e.g., in \cite{busing2005avoiding, borg2005modern}, our theory establishes the problem as well-posed under mild conditions --- even as it does not provide the practitioner with any insight on how to numerically solve the problem.   

\subsection{Notation}
For a positive integer $n$, $[n] := \{1,\dots,n\}$.
For a point $x \in \bbR^p$ and $r > 0$, $\ball(x,r)$ denotes the open ball centered at $x$ with radius $r$, and $\sphere(x,r)$ denotes the corresponding sphere.
The unit sphere will be denoted $\bbS^{p-1} := \{x \in \bbR^p : \|x\| = 1\}$.
For two distinct points $z,z' \in \bbR^p$, define $\hplane(z,z') = \{x : \|x-z\| = \|x-z'\|\}$, which is the affine hyperplane passing through $\frac12(z+z')$ perpendicular to $z-z'$; we also define $\hplane^+(z,z') = \{x : \|x-z\| < \|x-z'\|\}$, which is one of the two open half-spaces defined by that hyperplane.

\section{External Unfolding} 
\label{sec:external}

{\em External unfolding} is the problem of locating an individual in space based on preferences for some objects that are already positioned in that space. 
The earliest work that we know of in the literature is from \citet{gower1968adding}, who presents the problem as an out-of-sample extension of classical scaling. (We are talking about the method of \citet{young1938discussion}, later refined by \citet{torgerson1952multidimensional}, although Gower refers to his own work \cite{gower1966some}.)
We also know of contemporary work by \citet{carroll1967}, although only indirectly by way of \cite{srinivasan1973linear}.
While the term `external unfolding' is favored in Psychometrics \cite[Sec 16.1]{borg2005modern}, where it is often attributed to \citet{carroll1972individual}, the problem is best known in Engineering as `trilateration', `multilateration', or simply `lateration'  \cite{chrzanowski1965theoretical, glossary1994, fang1986trilateration, yang2009indoor, savvides2001dynamic, aspnes2006theory}.  
%The problem has also been considered in Applied Mathematics broadly understood \cite{kearsley1998solution, laurent2001polynomial, bakonyi1995euclidian, grone1984positive}.   
In Engineering, the objects are often called `anchors' or `landmarks'.
Some background is provided for statisticians in \cite{navidi1998statistical}.
While most of that work (in particular, within Engineering) has been done on the metric variant of the problem, the non-metric or ordinal variant has also received some attention, some of it being quite recent \cite{srinivasan1973linear, massimino2021you, davenport2013lost, anderton2019scaling}.
%We note that all these works place themselves in the point model. 

We consider what is known in Psychometrics as the point model, where the individuals and the objects are all represented by points in space \cite[Sec 14.1]{borg2005modern}. We do so in the non-metric or ordinal variant of the problem where an individual $x$ is to be located based on a ranking of the individual's preference for the objects. The objects are already embedded in space and preference is quantified in terms of the Euclidean distance.

\subsection{Discrete Setting}
\label{sec:external point discrete}
In the discrete (in fact, finite) setting --- which is the setting encountered in practice --- the problem can be described as follows: Given $y_1, \dots, y_n \in \bbR^p$ and a permutation $(r_1, \dots, r_n)$ of $(1, \dots, n)$,
\begin{equation}
\label{total_order_finite_point}
\text{Find $x \in \bbR^p$ such that $\|x-y_i\| < \|x-y_j\|$ whenever $r_i < r_j$.}
\end{equation}

The recent work by \citet{massimino2021you} provides the most comprehensive study to date. Even then, to simplify the analysis, the authors consider a variant of the problem where the design is random: objects are sampled iid from some isotropic normal distribution, yielding $y_1, \dots, y_n$ and $y'_1, \dots, y'_n$, and for an unknown point $x$ we have access to $\xi_i := \IND{\|x-y_i\| < \|x-y'_i\|}$ for all $i \in [n]$; that is, based on the pairs $(y_1, y'_1), \dots, (y_n, y'_n)$ and the comparisons $\xi_1, \dots, \xi_n$, the goal is to recover $x$. 
(This is the case in the noiseless setting. They also consider a noisy setting.) 
The assumption that the design is not only random but Gaussian is crucial to the analysis carried out in that paper.
In related work, \citet{pmlr-v97-canal19a} consider the problem of actively selecting the objects in order to maximize the accuracy in locating the individual. The authors provide an information lower bound for the problem and show that a Bayesian approach they propose matches that bound in order of magnitude.

\subsection{Continuum Setting}
\label{sec:external point continuum}
We are interested in whether the problem is constrained enough in order to recover the unknown location of the individual. This is clearly not the case in the discrete setting of \secref{external point discrete} as the solutions form an open set. However, we contend that the solution set reduces to a singleton in the limit of an infinite number of objects. Instead of tackling this claim in a frontal manner, in order to avoid technicalities and get to the core of the question, we follow \citet{shepard1966metric} and consider a limit model in the continuum where the set of objects is continuously infinite. 
%Given $\cY \subset \bbR^p$ and a function $r: \cY \to \bbR$,
%\begin{equation}
%\text{Find $x \in \bbR^p$ such that $\|x-y\| < \|x-y'\|$ whenever $r(y) < r(y')$.}
%\end{equation}
This corresponds to the limit of the discrete model if we imagine the points representing the objects $y_1, \dots, y_n$ as filling a set, denoted $\cY$ henceforth. 

We say that $x, x' \in \bbR^p$ are equivalent with respect to $\cY$ if  
\begin{equation}
\label{external_point_continuum}
\text{$\|x-y\| < \|x-y'\| \iff \|x'-y\| < \|x'-y'\|$, \quad for all $y, y' \in \cY$.}
\end{equation}
Indeed, any such $x$ and $x'$ are indistinguishable in terms of their preference for the objects in the set $\cY$.
The central question of whether an individual can be located based on its preference for the objects can be phrased as follows: {\em Do equivalent points coincide?} The answer is positive under some conditions on the set of objects. 
In fact, it is enough that the points be weakly equivalent in the sense that \begin{equation}
%\label{external_point_continuum_weak}
%\text{$\|x-y\| = \|x-y'\| \iff \|x'-y\| = \|x'-y'\|$, \quad for all $y, y' \in \cY$.}
\label{external_point_continuum_weak}
\begin{cases}
\|x-y\| < \|x-y'\| \Rightarrow \|x'-y\| \le \|x'-y'\|, \\
\|x-y\| \le \|x-y'\| \Leftarrow \|x'-y\| < \|x'-y'\|,
\end{cases}
\text{for all $y, y' \in \cY$.}
\end{equation}
We note that two points that are equivalent with respect to $\cY$ are also weakly equivalent, meaning that, if they satisfy \eqref{external_point_continuum}, they also satisfy \eqref{external_point_continuum_weak}.
The reason we work with the weaker condition \eqref{external_point_continuum_weak} is that we will refer to it later on.

%In fact, all we use is \eqref{external_point_continuum_weak}. 

\begin{theorem}
\label{thm:external_point}
If $\interior(\cY) \ne \emptyset$, weakly equivalent points must coincide. 
\end{theorem}

\begin{proof}
Take two weakly equivalent points, $x$ and $x'$.
By the fact that $\cY$ contains an open ball, it must also contain two open balls whose centers are not aligned with $x$, that is, there are $\cB_1 = \ball(z_1, t)$ and $\cB_2 = \ball(z_2, t)$, for some $z_1, z_2 \in \bbR^d$ and $t > 0$, such that $\cB_1, \cB_2 \subset \cY$ and $x, z_1, z_2$ are not collinear. 
We claim that $x' \in \cL_j := (xz_j)$ for $j = 1, 2$. If true, we can immediately conclude since $\cL_1 \cap \cL_2 = \{x\}$. 

We focus on proving that $x' \in \cL_1$. 
We use the fact that $\cL_1$ is the intersection of all the hyperplanes passing through $z_1$ with orthogonal direction perpendicular to $\cL_1$. Consider such a hyperplane $\cH$ with orthogonal direction given by the unit vector $v$. For $s \in \bbR$, define $y(s) = z_1 + s v$. Note that $y(s) \in \cY$ for any $s \in [-t, t]$, since in that case $y(s) \in \cB_1$. Also note that, by Pythagoras, $\|x-y(s)\|^2 = \|x-z_1\|^2 + s^2$, so that $\|x-y(s)\| < \|x-y(s')\|$ whenever $|s| < |s'|$. And by \eqref{external_point_continuum_weak} this implies that $\|x'-y(s)\| \le \|x'-y(s')\|$ whenever $|s| < |s'| < t$. Choosing $s_k = -t/2+1/k$ and $s'_k = t/2+1/k$, and letting $k\to\infty$, we get that $\|x'-y(-t/2)\| \le \|x'-y(t/2)\|$. And choosing $s_k = t/2-1/k$ and $s'_k = -t/2-1/k$, and letting $k\to\infty$, we get that $\|x'-y(t/2)\| \le \|x'-y(-t/2)\|$. We thus conclude that $\|x'-y(-t/2)\| = \|x'-y(t/2)\|$, which then implies that $x' \in \cH$ since $\hplane(y(-t/2), y(t/2)) = \cH$. 
Since this is true for any hyperplane $\cH$ that contains $\cL_1$, we have proved that $x' \in \cL_1$.
\end{proof}

\subsection{Discrete Asymptotic Setting}
\label{sec:external point asymptotic}

We now return to the discrete setting of \secref{external point discrete}, although in an asymptotic setting where the landmark points become dense in a suitable set. In doing so, we follow \citet{klein}, who also inspired our earlier work \cite{arias2017some}. 

Let $\cY_n := \{y_1, \dots, y_n\}$ denote the landmark points and $\cY_\infty := \{y_i: i \ge 1\}$. Guided by our analysis of the continuum model in \secref{external point continuum}, in particular \thmref{external_point}, our requirement below is that $\cY_\infty$ is dense somewhere,\footnote{This is the opposite of being {\em nowhere dense}, a well-known concept in topology.} by which we mean it is dense in some (non-empty) open set, or equivalently, its closure has non-empty interior.

\begin{theorem}
\label{thm:external point asymptotic}
In the present setting, suppose that $\cY_\infty$ is dense somewhere. For each $n$, let $(x_n, x'_n) \in \bbR^p \times \bbR^p$ be a pair of points that are equivalent with respect to $\cY_n$. Then any accumulation point of $\{(x_n, x'_n) : n \ge 1\}$ must be of the form $(x,x)$. In particular, if both $(x_n)$ and $(x'_n)$ converge, their limits must coincide. 
\end{theorem}

\begin{proof}
It suffices to consider the situation described at the end of the statement where both $(x_n)$ and $(x'_n)$ converge. We denote by $x$ and $x'$ their respective limits and turn to proving that $x = x'$. 

Let $\cY = \interior(\closure\, \cY_\infty)$. By our assumptions, $\cY$ is a non-empty open set and $\cY_\infty$ is dense in $\cY$.  \thmref{external_point}, it is enough to show that $x$ and $x'$ are weakly equivalent with respect to $\cY$. This results from a simple continuity argument. To be sure, take any $y, y' \in \cY$ such that $\|x-y\| < \|x-y'\|$. We need to prove that $\|x'-y\| \le \|x'-y'\|$. 
Take $0 < \delta < \frac14(\|x-y'\| - \|x-y\|)$.
Since $x_n \to x$ and $x'_n \to x'$, there is $n_0$ such that $\|x_n-x\| < \delta$ and $\|x'_n-x'\| < \delta$ for all $n \ge n_0$. 
Since $\cY_\infty$ is dense in $\cY$, there is $n \ge n_0$ and $n' \ge n_0$ such that $\|y_n-y\| < \delta$ and $\|y_{n'}-y'\| < \delta$. 
Let $n_1 = \max(n,n')$ and note that $\|x_{n_1}-x\| < \delta$ and $\|x'_{n_1}-x'\| < \delta$ since $n_1 \ge n_0$. 
By the triangle inequality,
\begin{align}
\|x_{n_1}-y_{n'}\| - \|x_{n_1}-y_n\| 
& \ge \|x-y'\| - \|x-y\| - 2 \|x_{n_1}-x\| - \|y_{n'}-y'\| - \|y_n-y\|,
\end{align}
so that 
\[
\|x_{n_1}-y_{n'}\| - \|x_{n_1}-y_n\| 
\ge \|x-y'\| - \|x-y\| - 4 \delta
> 0.
\]
Because $x_{n_1}$ and $x'_{n_1}$ are weakly equivalent with respect to $\cY_{n_1}$, this implies that
\[
\|x'_{n_1}-y_{n'}\| - \|x'_{n_1}-y_n\| \ge 0.
\]
By the triangle inequality,
\begin{align}
\|x'-y'\| - \|x'-y\| 
& \ge \|x'_{n_1}-y_{n'}\| - \|x'_{n_1}-y_n\| - 2 \|x'_{n_1}-x'\| - \|y_{n'}-y'\| - \|y_n-y\|,
\end{align}
so that
\begin{align}
\|x'-y'\| - \|x'-y\|
\ge - 4 \delta.
\end{align}
This being true for all sufficiently small $\delta>0$, we conclude. 
\end{proof}

%We say that $x, x' \in \bbR^p$ are equivalent with respect to $\cY$ if  
%\begin{equation}
%\label{external_point_continuum}
%\text{$\|x-y\| < \|x-y'\| \iff \|x'-y\| < \|x'-y'\|$, \quad for all $y, y' \in \cY$.}
%\end{equation}

\section{Multidimensional Scaling} 
\label{sec:mds}
In {\em multidimensional scaling (MDS)}, we have $n$ items, and the goal is to locate all them in space based on some pairwise dissimilarity information \cite{borg2005modern, cox2000multidimensional, young1987multidimensional}.  
%The problem is known under various names in various fields:  `embedding of metric spaces' in Mathematics and Computer Science \cite{blumenthal1938distance}; `Euclidean distance matrix completion' in Optimization \cite{laurent2001matrix}; `sensor network localization' in Engineering \cite{priyantha2003anchor}. 
%MDS is an integral part of Multivariate Analysis \cite{MVA, seber2009multivariate}.

We focus on the non-metric or ordinal variant of the problem, which has been of great interest in Psychometrics and beyond.
The body of work on this problem is substantial. Methodological work was pioneered by \citet{MR0173342,MR0140376} and \citet{kruskal1964nonmetric,kruskal1964multidimensional}, and has continued to this day \cite{terada2014local, mair2022more, anderton2019scaling, liu2016towards, agarwal2007generalized, van2012stochastic, tamuz2011adaptively}. 
Not much theoretical work is available. \citet{shepard1966metric} pioneered some theory, which was elaborated much more recently by \citet{klein} and ourselves \cite{arias2017some}. 
\citet{jain2016finite} consider a situation where the available information is noisy or imprecise and approach the problem via maximum likelihood. 

In our study of the problem below, we focus on what \citet{torgerson1952multidimensional} calls the (complete) method of triads, and what is referred to as triadic comparisons in \cite{DELEEUW1982285}. This model goes back to early work of \citet{stumpf1880} in the late 1800s. 

\subsection{Discrete Setting}
\label{sec:mds discrete}
In the usual setting, ordinal MDS can be described as follows: 
Given a set of row ranks $(r_{ij} : i, j \in [n])$, with each $(r_{i1}, \dots, r_{in})$ being a permutation of $(1, \dots, n)$, and a dimension $p \ge 1$, 
\begin{equation}
\label{mds_discrete}
\text{Find $x_1, \dots, x_n \in \bbR^p$ such that $\|x_i-x_j\| < \|x_i-x_k\|$ whenever $r_{ij} < r_{ik}$.}
\end{equation}
We note that any similarity transformation of a solution is also a solution.

In regards to the same foundational question of uniqueness (up to a similarity transformation), except for some earlier work in dimension $p=1$ \cite{suppes1955axiomatization, aumann1958coefficients}, the first meaningful contribution appears to be that of \citet{shepard1966metric}. To simplify matters, Shepard considered a continuous limit model, as is routinely done in Physics, for example. We describe such a model in the following subsection, and elaborate on his reasoning. This is the model that inspired our work.
Further progress was made decades later by \citet{klein}, who considered the finite sample situation described above and derived conditions under which, in the asymptotic limit $n\to\infty$ where the points fill a subset of $\bbR^p$, the solutions are constrained to be similitudes of each other. 
We later refined their results in \cite{arias2017some}.

\subsection{Continuous Setting}
\label{sec:mds continuum}
Following \citet{shepard1966metric}, we consider a limit model in the continuum where the set of items forms an uncountably infinite subset $\cX \subset \bbR^p$.
%We note that the techniques used in \cite{klein, arias2017some} can be applied here to obtain a comparable result in the finite sample setting. 

In the discrete setting, two configurations in dimension $p$, $\{x_1, \dots, x_n\}$ and $\{x'_1, \dots, x'_n\}$, are indistinguishable if it holds that
\begin{equation}
\label{mds_discrete_equiv}
\text{$\|x_i-x_j\| < \|x_i-x_k\| \iff \|x'_i-x'_j\| < \|x'_i-x'_k\|$, \quad for all $i,j,k \in [n]$.}
\end{equation}
In order to transition from the discrete setting to the continuous setting, we consider these configurations as being in correspondence via the function $x_i \mapsto x'_i$ defined on $\{x_1, \dots, x_n\}$. 
When considering equivalent configurations in the continuum, we are thus led to study injective functions $f : \cX \to \bbR^p$ satisfying 
\begin{equation}
\label{mds_continuum}
\text{$\|x-x'\| < \|x-x''\| \iff \|f(x)-f(x')\| < \|f(x)-f(x'')\|$, \quad for all $x, x', x'' \in \cX$.}
\end{equation}
We note that this property is equivalent to 
\begin{equation}
\label{mds_continuum_set}
f(\ball(x, \|x-x'\|) \cap \cX) = \ball(f(x), \|f(x)-f(x')\|)) \cap f(\cX), \quad \forall x,x' \in \cX. 
\end{equation}
We note that \eqref{mds_continuum} implies
\begin{equation}
\label{mds_continuum_equal}
\text{$\|x-x'\| = \|x-x''\| \iff \|f(x)-f(x')\| = \|f(x)-f(x'')\|$, \quad for all $x, x', x'' \in \cX$,}
\end{equation}
which is itself equivalent to
\begin{equation}\label{mds_continuum_equal_set}
f(\sphere(x, \|x-x'\|) \cap \cX) = \sphere(f(x), \|f(x)-f(x')\|)) \cap f(\cX), \quad \forall x,x' \in \cX.
\end{equation}

We first state the result in the same setting that \citet{shepard1966metric} considers, which crucially assumes that $f$ is a bijection of the entire space. We note that Shepard does that implicitly, even though being bijective is not an immediate consequence of being weakly isotonic.

\begin{proposition}
\label{prp:shepard}
Suppose that $\cX = \bbR^p$. Then any {\em bijective} function $f: \bbR^p \to \bbR^p$ satisfying \eqref{mds_continuum} must be a similarity transformation. 
\end{proposition}

When $f$ is surjective, \eqref{mds_continuum_equal_set} takes the form
\begin{equation}\label{f_shepard}
f(\sphere(x, \|x-x'\|)) = \sphere(f(x), \|f(x)-f(x')\|), \quad \forall x,x' \in \bbR^p.
\end{equation}
This implies that $f$ transforms any sphere into a sphere.
\citet{shepard1966metric} argues from there that ``every sphere-preserving transformation is either a similarity transformation or the product of an inversion (in a sphere) and an isometry", citing\footnote{We note that, strictly speaking, the result is stated for dimension three, and earlier in the same book, for dimension two. These are Theorems~6.71 and~7.71 in the 1969 edition of the book.} \cite[p 104]{coxeter1961introduction}.
He then goes on to say that ``The possibility of an inversive transformation can immediately be ruled out, however. It preserves neither the rank order of concentric spheres nor the equality of nonconcentric spheres, whereas both of these invariances are required by the given rank order of the interpoint distances." 

\medskip
We let Shepard's arguments stand on their own, and turn our attention to establishing a more general result using a different approach, following \cite{arias2017some} instead.
It turns out that all that is needed of $f$ is that it be {\em weakly isotonic} in the sense that
\begin{equation}
\label{weakly_isotonic}
\text{$\|x-x'\| < \|x-x''\| \implies \|f(x)-f(x')\| \le \|f(x)-f(x'')\|$, \quad for all $x, x', x'' \in \cX$.}
\end{equation}
(This is the definition given in \cite{arias2017some}. A slightly stronger notion is used in \cite{klein}, where the inequality on the right-hand side is strict.)
The following result extends existing ones in \cite{klein, arias2017some} to situations where $\cX$ is not necessarily open or connected.

\begin{theorem}
\label{thm:mds}
Suppose that $\interior\,\cX \ne \emptyset$. Then a weakly isotonic function on $\cX$ must be a similarity transformation.\footnote{When we say that a function $f : \cX \subset \bbR^d \to \bbR^d$ is a similarity transformation, we mean that it coincides on $\cX$ with a similarity transformation defined on the entirety of $\bbR^d$.}
\end{theorem}

\begin{proof}
%\cite[Prop~7]{klein}, and also \cite[Th~1]{arias2017some}, imply the result when $\cX$ is connected as well. 
If $\cX$ is an open ball, then the result is a special case of \cite[Prop~7]{klein} or \cite[Th~1]{arias2017some}.
%We sketch the arguments following the latter reference. The line of reasoning is as follows: (1) a weakly isotonic function on such a set $\cX$ is continuous; (2) a weakly isotonic function on a convex set preserves midpoints, i.e., satisfies Jensen's functional equation; 
%%\begin{equation}\label{jensen}
%%g(\tfrac12(x+x')) = \tfrac12 g(x) + \tfrac12 g(x'), \quad \forall x,x'.
%%\end{equation}
%note that (1) and (2) combined yield that a weakly isotonic function on an open ball is affine (i.e., coincides on its domain with an affine function); and we conclude with the fact that (3) an affine function that satisfies \eqref{mds_continuum_equal} on a ball is a similitude (i.e., coincides on its domain with a similarity transformation). 
%We refer the reader to \cite{arias2017some}, and also \cite{klein}, for details.

Now, to prove the theorem as stated, let $\cB$ be any open ball contained in $\cX$ and let $f$ be weakly isotonic on $\cX$.
We know that $f$ coincides on $\cB$ with a similarity transformation. Without loss of generality, we may assume this similarity transformation to be the identity function, in which case $f(x) = x$ for all $x \in \cB$. 
In particular, $f(\cB) = \cB$, and via \eqref{weakly_isotonic}, we have, for any $x \in \cX$,
\begin{equation}
%\label{internal_point_continuum}
\text{$\|x-x'\| < \|x-x''\| \implies \|f(x)-x'\| \le \|f(x)-x''\|$, \quad for all $x', x'' \in \cB$.}
\end{equation}
Therefore, $x$ and $f(x)$ are weakly equivalent with respect to $\cB$ in the sense of \eqref{external_point_continuum_weak}, and by \thmref{external_point}, this implies that $f(x) = x$. And this is true for any $x \in \cX$, so that $f$ coincides with the identity function on the entirety of $\cX$.
\end{proof}
%\begin{lemma}\label{lem:jensen}
%Let $\cB$ be an open ball of $\bbR^p$ and suppose $f: \cB \to \bbR^p$ is bounded and satisfies Jensen's functional equation:
%\begin{equation}\label{jensen}
%f(\tfrac12(x+x')) = \tfrac12 f(x) + \tfrac12 f(x'), \quad \forall x,x' \in \cB.
%\end{equation}
%Then $f$ coincides with an affine transformation on $\cB$.
%\end{lemma}
%\begin{proof}
%
%\end{proof}

\subsection{Discrete Asymptotic Setting}
\label{sec:mds point asymptotic}

We return to the discrete setting of \secref{mds discrete}. Specifically, we consider an asymptotic setting where there is a set of points satisfying the ordinal information provided, meaning satisfying \eqref{mds_discrete}, that becomes dense in a suitable domain. This is the framework of \cite{klein, arias2017some}. By leveraging our analysis of the continuum model, we are able to operate under weaker assumptions compared to the work there: We do away with the openness and connectedness assumptions. (The boundedness assumption is inconsequential.) 

\begin{theorem}
\label{thm:mds asymptotic}
Let $(x_n)$ and $(x'_n)$ be two bounded sequences of distinct points in $\bbR^d$ such that, for each $n$, $\{x_1, \dots, x_n\}$ and $\{x'_1, \dots, x'_n\}$ are equivalent configurations in the sense of \eqref{mds_discrete_equiv}. In addition, assume that $(x_n)$ is dense somewhere and that both sequences are bounded.   
For each $n$, let $f_n : \{x_i: i \ge 1\} \to \{x'_i: i \ge 1\}$ such that $f_n(x_i) = x'_i$ for $i \in [n]$. Then $(f_n)$ is sequentially compact for the pointwise convergence topology and all the functions where it accumulates are similarity transformations.
\end{theorem}

\begin{proof}
The fact that $(f_n)$ is sequentially compact for the pointwise convergence topology is a standard result in topology that relies on the so-called {\em diagonal process}. See \cite[Lemma 2]{arias2017some}. 

To establish the second part of the statement, it suffices to consider the situation where $(f_n)$ converges pointwise on $\cX_\infty := \{x_i : i \ge 1\}$. Let $f$ denote the limit, so that $\lim_n f_n(x_i) = f(x_i)$ for all $i \ge 1$. It is clear that $f$ is weakly isotonic on $\cX_\infty$. Let $\cB$ be an open ball where $\cX_\infty$ is dense. By \cite[Lemma 4]{arias2017some}, $f$ is uniformly continuous on $\cB \cap \cX_\infty$, and therefore admits a unique continuous function on $\cB$, also denoted by $f$. This extension is weakly isotonic on $\cX := \cB \cup \cX_\infty$, so that by \thmref{mds}, $f$ coincides on $\cX$ --- and therefore on $\cX_\infty$ --- with a similarity transformation. 
\end{proof}

\begin{remark}
Uniform convergence are established in \cite{klein, arias2017some}, but under the assumption that the limiting set be included and dense in some connected open set. We believe that, with (much) more work, a uniform rate of convergence can be obtained in the more general setting considered here. 
\end{remark}

\section{Internal Unfolding} 
\label{sec:internal}

In {\em internal unfolding}, we have $m$ individuals expressing preferences for $n$ objects, and the goal is to locate the $m$ individuals and the $n$ objects in space. Unlike in external unfolding (\secref{external}), the location of the objects is unknown.
The origins of internal unfolding in the Psychometrics literature date back to \citet{coombs1950psychological}, at least for the point model. 
The problem is also known as `multidimensional unfolding', or just `unfolding'  \cite{borg2005modern}.  

We consider non-metric or ordinal variant of the problem, with a focus on the {\em conditional} setting where each individual ranks the objects in order of preference without first agreeing with other individuals on some ordinal scale. This corresponds to the method of triads. 

%It is well-known in the field that the problem is numerically difficult as the iterative method otherwise popular in the metric setting often returns degenerate solutions \cite{busing2005avoiding, borg2005modern}.

%\subsection{Point Model}
%\label{sec:internal_point}

\subsection{Discrete Setting}
\label{sec:internal point discrete}
The practitioner adopting the point model is confronted with the following problem: 
Given a set of row ranks $(r_{ik} : i \in [m], k \in [n])$, with each $(r_{i1}, \dots, r_{in})$ being a permutation of $(1, \dots, n)$, and a dimension $p \ge 1$, 
\begin{equation}\label{internal point problem}
\text{Find $x_1, \dots, x_m \in \bbR^p$ and $y_1, \dots, y_n \in \bbR^p$ such that $\|x_i-y_k\| < \|x_i-y_l\|$ whenever $r_{ik} < r_{il}$.}
\end{equation}
We note that any similarity transformation of a solution --- applied to both the individuals and the objects --- is also a solution.

\citet{bennett1960multidimensional} propose three methods for determining the smallest dimension where an exact embedding can be realized, and in the process offer some elementary observations on things like the number of isotonic regions (aka Voronoi cells).
A follow-up paper by the same authors \cite{hays1961multidimensional} considers extending the basic approach developed by \citet{coombs1950psychological} for the case of dimension $p=1$ to general $p > 1$ by studying the order of individuals when projected onto lines. 
This combinatorial and basic geometrical work was developed further by \citet{davidson1973geometrical, davidson1972geometrical}.
Beyond that, the only theoretical results we are aware of pertain to the study of degenerate solutions \cite{busing2005avoiding, davidson1973geometrical, de1983}.
%(Some theory is developed in \cite{chen2021unfolding} for a binary preference model that also falls under the umbrella name of multidimensional unfolding --- but the model is a very different.)

\subsection{Continuous Setting}
\label{sec:internal point continuum}
Again inspired by \citet{shepard1966metric}, we consider a limit model in the continuum where the number of individuals and the number of objects are both infinite, represented by subsets $\cX \subset \bbR^p$ and $\cY \subset \bbR^p$, respectively. 

Staying with the discrete model for a moment, two configurations in dimension $p$, one of them $\{x_1, \dots, x_m; y_1, \dots, y_n\}$ and the other $\{x'_1, \dots, x'_m; y'_1, \dots, y'_n\}$, are indistinguishable if it holds that
\begin{equation}
\label{internal_equivalent_point_discrete}
\text{$\|x_i-y_k\| < \|x_i-y_l\| \iff \|x'_i-y'_k\| < \|x'_i-y'_l\|$, \quad for all $(i,k,l) \in [m] \times [n] \times [n]$.}
\end{equation}
In preparation to pass to the continuum, we regard these configurations as being in correspondence via the pair of functions $x_i \mapsto x'_i$ and $y_k \mapsto y'_k$, defined on $\{x_1, \dots, x_m\}$ and $\{y_1, \dots, y_n\}$, respectively. 
When considering equivalent configurations in the continuum, we are thus led to study pairs of injective functions $f : \cX \to \bbR^p$ and $g: \cY \to \bbR^p$ satisfying 
\begin{equation}
\label{internal_point_continuum}
\text{$\|x-y\| < \|x-y'\| \iff \|f(x)-g(y)\| < \|f(x)-g(y')\|$, \quad for all $(x, y, y') \in \cX \times \cY \times \cY$.}
\end{equation}
This is equivalent to 
\begin{equation}
\label{internal_point_continuum_set}
g(\ball(x, \|x-y\|) \cap \cY) = \ball(f(x), \|f(x)-g(y)\|)) \cap g(\cY), \quad \text{for all }(x,y) \in \cX \times \cY. 
\end{equation}
We note that \eqref{internal_point_continuum} implies
\begin{equation}
\label{internal_point_continuum_equal}
\text{$\|x-y\| = \|x-y'\| \iff \|f(x)-g(y)\| = \|f(x)-g(y')\|$, \quad for all $(x, y, y') \in \cX \times \cY \times \cY$,}
\end{equation}
which is itself equivalent to
\begin{equation}\label{internal_point_continuum_equal_set}
g(\sphere(x, \|x-y\|) \cap \cY) = \sphere(f(x), \|f(x)-g(y)\|)) \cap g(\cY), \quad \forall (x,y) \in \cX \times \cY.
\end{equation}

We first establish the result in the setting that we believe \citet{shepard1966metric} would have considered.
%Note that \eqref{g_shepard} below is stronger than \eqref{internal_point_continuum_equal_set}, and while the latter follows from \eqref{internal_point_continuum}, it is not immediately clear that the former does. 

\begin{proposition}
\label{prp:shepard_internal}
In the situation where $\cX = \cY = \bbR^p$, consider any pair of injective functions $(f,g)$ satisfying \eqref{internal_point_continuum} such that $g(\bbR^p) = \bbR^p$. Then $f = g = L$ for some similarity transformation $L$. 
\end{proposition}

With the additional assumption that $g(\bbR^p) = \bbR^p$, meaning that $g$ is not only injective but also surjective, \eqref{internal_point_continuum_equal_set} becomes
\begin{equation}\label{g_shepard}
g(\sphere(x, \|x-y\|)) = \sphere(f(x), \|f(x)-g(y)\|)), \quad \forall (x,y) \in \bbR^p \times \bbR^p.
\end{equation}

\begin{lemma} \label{lem:fffggg}
Consider any pair of injective functions $(f,g)$ satisfying \eqref{g_shepard}. If $x, x', x'' \in \cX$ and $y, y', y'' \in \cY$ are all collinear, then so are $f(x), f(x'), f(x''), g(y), g(y'), g(y'')$.
\end{lemma}

\begin{proof}
Assume without loss of generality that the points are all distinct. 

We first show that $f(x), f(x'), g(y)$ are collinear.
Indeed, since $x,x',y$ are collinear, and we just assumed that $x \ne x'$, it must be that $y$ is the only point at the intersection of $\cS := \sphere(x, \|x-y\|)$ and $\cS' := \sphere(x', \|x'-y\|)$.
Now, by \eqref{g_shepard}, 
\begin{align}
g(\cS) = \sphere(f(x), \|f(x)-g(y)\|), && g(\cS') = \sphere(f(x'), \|f(x')-g(y)\|),
\end{align}
and, by the fact that $g$ is injective, these two spheres only have one point in common, $g(y)$, and so they must be tangent as well. This then implies that their centers, $f(x)$ and $f(x')$, are collinear with their point of contact, $g(y)$. 

By the same token, $f(x), f(x'), g(y')$, $f(x), f(x'), g(y'')$, and also $f(x), f(x''), g(y)$, must be collinear. And from all this, we are able to conclude.
\end{proof}

\begin{proof}[Proof of \prpref{shepard_internal}]
It suffices to show that $f$ and $g$ coincide, as we can then deduce from \eqref{internal_point_continuum} that $f$ is weakly isotonic, and is therefore a similarity transformation via \thmref{mds}. 

Take any $z_0$. We want to show that $f(z_0) = g(z_0)$.
Consider $z_1,z'_1, z_2,z'_2$ on some sphere $\cS$ centered at $z_0$ such that $(z_1z'_1)$ and $ (z_2z'_2)$ intersect at $z_0$, so that $[z_1z'_1]$ and $[z_2z'_2]$ are diameters of the sphere. 
By \eqref{g_shepard}, $g(\cS)$ is a sphere centered at $f(z_0)$ and passing through $g(z_1),g(z'_1),g(z_2),g(z'_2)$. 
And, by \lemref{fffggg}, $f(z_0),g(z_1),g(z'_1)$ are collinear, and so are $f(z_0),g(z_2),g(z'_2)$, implying that $[g(z_1)g(z'_1)]$ and $[g(z_2)g(z'_2)]$ are diameters of $g(\cS)$. By the fact that $g$ is injective, we have that $g(z_1),g(z'_1),g(z_2),g(z'_2)$ are distinct, so that $f(z_0)$ is the only point at the intersection of the lines $(g(z_1)g(z'_1))$ and $(g(z_2)g(z'_2))$. However, \lemref{fffggg} also gives that $g(z_0),g(z_1),g(z'_1)$ are collinear, and that $g(z_0),g(z_2),g(z'_2)$ are collinear, implying in the same way that $g(z_0)$ is also at the intersection of these two lines, forcing $g(z_0) = f(z_0)$.
%We have thus established that $f = g$, everywhere. 
\end{proof}

%\ery{changes between here and the theorem, showing that \eqref{gBopen} is enough and reorganizing things}

It is possible to weaken the assumptions substantially. 
%Consider the following conditions:
%\begin{equation}
%\label{XcapY}
%\text{$\cX \cap \cY$ has non-empty interior};
%\end{equation}
%\begin{equation}
%\label{conY}
%\text{$\cX$ contains the convex hull of $\cY$};
%\end{equation}
%\begin{equation}
%\label{Ycon}
%\text{$\cY$ is convex}.
%\end{equation}
The following result generalizes \prpref{shepard_internal}.
It relies on the following weaker variant of \eqref{internal_point_continuum},
\begin{equation}
\label{internal_point_continuum_weak}
\text{$\|x-y\| < \|x-y'\| \implies \|f(x)-g(y)\| \le \|f(x)-g(y')\|$, \quad for all $(x, y, y') \in \cX \times \cY \times \cY$.}
\end{equation}

\begin{proposition}
\label{prp:gBopen}
Suppose $\interior(\cX \cap \cY) \ne \emptyset$, and consider any pair of injective functions $(f,g)$ satisfying \eqref{internal_point_continuum_weak}. Assume that
\begin{equation}
\label{gBopen}
\text{There is $\cB \subset \cX \cap \cY$ open such that either $g(\cB)$ is open, or $g(\cY)$ is dense in $f(\cB)$.}
\end{equation}
Then $f$ coincides on $\cX$ with a similarity transformation. If in addition either $\cX$ contains the convex hull of $\cY$, or $\cY$ is convex, then $f$ and $g$ coincide with the same similarity transformation on $\cX$ and $\cY$, respectively. 
\end{proposition}

\begin{proof}
In view of \lemref{ball}, it suffices to show that $f = g$ on $\cB$.
Assume for contradiction that there is $x \in \cB$ such that $f(x) \ne g(x)$. By \eqref{internal_point_continuum_weak}, 
\begin{equation}
\label{gBopen proof 1}
0< r := \|f(x)-g(x)\| \le \|f(x)-g(y)\|, \quad \forall y \in \cY.
\end{equation}
In particular, this prevents $g(\cY)$ from being dense in $f(\cB)$ since \eqref{gBopen proof 1} implies that $B(f(x), r) \cap g(\cY) = \emptyset$ . 
It also implies that no ball centered at $g(x)$ can be contained inside $g(\cB)$, preventing $g(\cB)$ from being open.
\end{proof}

\begin{lemma}
\label{lem:ball}
Consider any pair of injective functions $(f,g)$ satisfying \eqref{internal_point_continuum_weak} and such that $f = g$ on a (non-empty) open ball contained in $\cX \cap \cY$. Then $f$ coincides on $\cX$ with a similarity transformation. If in addition either $\cX$ contains the convex hull of $\cY$, or $\cY$ is convex, then $f$ and $g$ with the same similarity transformation on $\cX$ and $\cY$, respectively. 
\end{lemma}

\begin{proof}
Let $\cB$ denote that open ball, which we assume to be the unit ball without loss of generality. 

By way of \eqref{internal_point_continuum_weak}, we have that $f$ is weakly isotonic on $\cB$ in the sense of \eqref{weakly_isotonic}. Therefore, by \thmref{mds}, $f$ coincides on $\cB$ with a similarity transformation, which we take to be the identity without loss of generality so that $f(x) = x$ for all $x \in \cB$.

Now, take any $x\in\cX$. 
By \eqref{internal_point_continuum_weak} and the fact that $f = g = {\rm id}$ on $\cB$,   
\[\|x-y\| < \|x-y'\| \implies \|f(x)-y\| \le \|f(x)-y'\|, \quad \forall y,y' \in \cB,\]
so that $x$ and $f(x)$ are weakly equivalent with respect to $\cB$ in the sense of \eqref{external_point_continuum_weak}. We may thus apply \thmref{external_point} to obtain that $x$ and $f(x)$ coincide. This being true for any $x \in \cX$, we have established the first part of the statement, that $f$ coincides on $\cX$ with a similarity. 

We proceed and continue to assume that $f(x) = x$ for all $x \in \cX$. We need to prove that $g(y) = y$ for all $y \in \cY$. At this point, we know that $g$ satisfies
\begin{equation}
\label{ball_proof1}
\begin{cases}
\|x-y\| < \|x-y'\| &\Rightarrow \|x-g(y)\| \le \|x-y'\|, \\
\|x-y\| > \|x-y'\| &\Rightarrow \|x-g(y)\| \ge \|x-y'\|,
\end{cases}
\quad \text{for all $x,y,y' \in \cX \times \cY \times \cB$.}
\end{equation}

• {\em Suppose $\cX$ contains the convex hull of $\cY$.}
Take any $y \in \cY$ not in $\cB$, for otherwise we already know that $g(y) = y$. 
Let $u := y/\|y\|$, and for $s \in (-1,1)$, define $y(s) := s u$ and $x(s) = \frac12(y(s)+y)$. The rationale for the notation is that, in the range of $s$ considered, $y(s) \in \cB$, and by the fact that $\cX$ contains the convex hull of $\cY$, $x(s) \in \cX$. Let $r(s) = \|x(s) - y\| = \frac12(\|y\|-s)$. Noting that $y$ is the only point such that $\|x(-1/2) - y\| = r(-1/2)$ and $\|x(1/2) - y\| = r(1/2)$, it suffices to show that $\|x(-1/2) - g(y)\| = r(-1/2)$ and $\|x(1/2) - g(y)\| = r(1/2)$. We focus on the latter.
First, when $s < 1/2$, we have $\|x(1/2) - y\| < \|x(1/2) - y(s)\|$, and an application of \eqref{ball_proof1} gives $\|x(1/2) - g(y)\| \le \|x(1/2) - y(s)\|$. Letting $s \nearrow 1/2$ gives $\|x(1/2) - g(y)\| \le r(1/2)$. Similarly, when $s > 1/2$, we have $\|x(1/2) - y\| > \|x(1/2) - y(s)\|$, and an application of \eqref{ball_proof1} gives $\|x(1/2) - g(y)\| \ge \|x(1/2) - y(s)\|$. Letting $s \searrow 1/2$ gives $\|x(1/2) - g(y)\| \ge r(1/2)$. 
%Together, we conclude  that $\|x(1/2) - g(y)\| = r(1/2)$.

• {\em Suppose $\cY$ is convex.}
We proceed by induction. Suppose we know that $g = {\rm id}$ on $\cB_m \cap \cY$  for some $m \ge 1$ integer, where $\cB_m := \ball(0, m)$. Note that this is true for $m = 1$ by assumption, since $\cB_1 = \cB$. Take any $y \in (\cB_{m+1}\setminus \cB_m) \cap \cY$.
We use the same notation: $u := y/\|y\|$, and for $s \in \bbR$, $y(s) := s u$, $x(s) = \frac12(y(s)+y)$, and $r(s) = \|x(s) - y\|$.
In particular, $y = y(m+a)$ for some $0 \le a < 1$.
We consider $s \in (-m, m)$ to ensure that $y(s) \in \cB_m$, and also that $y(s) \in \cY$ by the fact that $y(s) \in [0y]$ and $\cY$ is convex. (The origin, 0, is in $\cY$ since $0 \in \cB \subset \cY$.)  
By our induction hypothesis, we thus have $g(y(s)) = y(s)$.
We further restrict $s$ to be in the range $-m < s < 2-a-m$, to also ensure that $x(s) \in \cB$, so that $x(s) \in \cX$. The arguments are now the same and center on the fact that $y$ is the only point such that $\|x(1/2-m) - y\| = r(1/2-m)$ and $\|x(1-m) - y\| = r(1-m)$.
\end{proof}

Although \prpref{gBopen} generalizes \prpref{shepard_internal}, it remains somewhat unsatisfactory as it still makes an assumption on the alternative configuration $f(\cX) \times g(\cY)$ by way of \eqref{gBopen}. It turns out that this assumption is not needed, although its absence makes the situation substantially more complicated. 

\begin{theorem} \label{thm:internal_point}
Suppose $\interior(\cX \cap \cY) \ne \emptyset$, and consider any pair of injective functions $(f,g)$ satisfying \eqref{internal_point_continuum}.
Then $f$ coincides on $\cX$ with a similarity transformation. If in addition either $\cX$ contains the convex hull of $\cY$, or $\cY$ is convex, then $f$ and $g$ coincide on $\cY$. 
\end{theorem}

The proof occupies the rest of the section. Until the end of the proof, $(f,g)$ denotes a pair of injective functions satisfying \eqref{internal_point_continuum}. In view of \lemref{ball}, it suffices to consider a situation where $\cX = \cY$ is an open ball, which we denote $\cB$ henceforth, and to prove that $f$ and $g$ coincide on $\cB$.

For a subset $\cS \subset \bbR^p$ we define $\dim \cS$ to be the dimension of the affine space spanned by $\cS$, denoted $\Span \cS$. We denote by $\dir \cS$ the direction of $\Span\cS$ and by $\Vect \cS$ the vector space spanned by $\cS$. In particular, for any $s \in \cS$, we have $\dir \cS = \Vect(\cS-s)$ and $\Span\cS = s + \dir\cS$. We say that two subsets $\cS$ and $\cS'$ are parallel if $\dir \cS \subset \dir \cS'$ or $\dir \cS' \subset \dir \cS$.

We already saw that any pair of functions $(f,g)$ satisfying \eqref{internal_point_continuum} also satisfies \eqref{internal_point_continuum_set}, which here takes the form 
\begin{equation} \label{easy-iii}
g(\ball(x,\|y-x\|)\cap \cB) = \ball(f(x),\|g(y) - f(x)\|) \cap g(\cB), \quad \forall x,y \in \cB.
\end{equation}
It is also the case that 
\begin{equation} \label{easy-i}
f(\hplane^+(y,y') \cap \cB) = \hplane^+(g(y),g(y')) \cap f(\cB), \quad \forall y \ne y' \in \cB;
\end{equation}
\begin{equation} \label{easy-ii}
f(\hplane(y,y') \cap \cB) = \hplane(g(y),g(y')) \cap f(\cB), \quad \forall y \ne y' \in \cB.
\end{equation}
Recall that $\hplane(y,y')$ is the hyperplane going through the midpoint of, and orthogonal to the line segment $[yy']$, while $\hplane^+(y,y')$ is the half-space with boundary $\hplane(y,y')$ and containing $y$.
%We also saw in \thmref{external_point} that, if $\cY$ has non-empty interior, $f$ is injective. 
%In addition, $g$ is injective on $(\dir \cX) \cap \cY$. To see this, assume $g(y) = g(y')$ for some $y, y' \in \cY$. Then according to \eqref{internal_point_continuum}, $\|y'-x\| = \|y-x\|$ for all $x \in \cX$, implying $
%\|y'\|^2-\|y\|^2 = 2\inner{x}{y'-y}$ for all $x \in \cX$, in turn implying that $\inner{z}{y'-y} = 0$ for all $z \in \dir \cX$.

\begin{lemma} \label{lem:fdimb} 
For any subset $\cS \subset \cB$, $\dim f(\cS) = \dim \cS$.
\end{lemma}

\begin{proof} 
We first prove that $\dim f(\cS) \geq \dim \cS$. 
Let $k=\dim \cS$ and let $x_1,\dots,x_{k+1}$ be affinely independent points of $\cS$. Assume that $\dim \{f(x_1),\dots,f(x_{k+1})\} < k$. Because the Vapnik--Chervonenkis dimension of affine hyperplanes in $\bbR^k$ is exactly $k+1$, the set $ \{f(x_1),\dots,f(x_{k+1})\}$ cannot be shattered, and we can find a subset $I \subset [k+1]$ such that $(f(x_i))_{i\in I}$ cannot be separated from $(f(x_i))_{i\in I^c}$. But because $x_1,\dots,x_{k+1}$ is affinely independent, it is shattered by affine hyperplanes and it exists $\cH$ that separates $(x_i)_{i\in I}$ from $(x_i)_{i \in I^c}$. Now we can find $y,y' \in \cB$ such that $\cH = \hplane(y,y')$: indeed, it must be the case that both $I$ and $I^c$ are not empty, thus $\cH$ operates a non-trivial separation of $\{x_1,\dots,x_{k+1}\}$ and so intersects $\cB$; since $\cB$ is open, it is then easy to find two such points $y,y' \in \cB$. Now it is straight-forward to see that $\hplane(g(y),g(y'))$ also separates $(f(x_i))_{i\in I}$ from $(f(x_i))_{i \in I^c}$ through \eqref{easy-i}, leading to a contradiction. Thus $\dim f(\cS) \geq \dim \{f(x_1),\dots,f(x_{k+1})\} \geq k = \dim \cS$.

We now prove that $\dim f(\cS) \leq \dim \cS$. We do so by descending induction on $\dim \cS$.
When $\dim\cS = p - 1$, then there exists distinct $y,y' \in \cB$ such that $\cS \subset \hplane(y,y')$ and \eqref{easy-i} yields that $f(\cS) \subset \hplane(g(y),g(y'))$ hence $\dim f(\cS) \leq p-1$ by injectivity of $g$. 
Now by induction, if $\dim \cS \leq p-2$, then at least $\dim f(\cS) \leq p - 1$. Now assume that $\dim f(\cS \cup \{x\}) = \dim f(\cS)$  for all $x \notin \Span \cS$. Then, according to the first part of the proof,
$$
\dim f(\cS) = \dim f(\cS \cup (\cB \setminus \Span \cS)) \geq \dim(\cB \setminus \Span \cS) = p,
$$
which is absurd. Therefore, there exists $x \notin \Span \cS$ such that $\dim f(\cS \cup \{x\}) = \dim f(\cS) + 1$. By induction, we get
$$
\dim f(\cS) = \dim f(\cS \cup \{x\}) - 1 \leq \dim (\cS \cup \{x\}) - 1 = \dim \cS.
$$
which ends the proof.
\end{proof}

\begin{lemma} \label{lem:rect} 
Let $\cR$ be any $2^p$-tuple of $\cB$ forming a hyperrectangle. Then $g(\cR)$ is also a $2^p$-tuple forming a hyperrectangle in the same configuration as $\cR$.
\end{lemma}

\begin{proof} 
It suffices to establish that, for any point set $y_1, y_2, y_3, y_4 \in \cB$ that forms a rectangle, the point set $g(y_1), g(y_2), g(y_3), g(y_4)$ also forms a rectangle. 

Assume without loss of generality that $y_1-y_2 = y_3 - y_4$.
In that case, we have $\hplane(y_1,y_2) = \hplane(y_3,y_4)$.
Then, by \eqref{easy-ii}, $\hplane(g(y_1),g(y_2))$ and $\hplane(g(y_3),g(y_4))$ both contain $f(\hplane(y_1,y_2) \cap \cB)$, and because that subset has dimension $p-1$ by \lemref{fdimb}, it must be that $\hplane(g(y_1),g(y_2)) = \hplane(g(y_3),g(y_4)) =: \cH$.
In particular, $v_{12} := g(y_1)-g(y_2)$ is parallel to $v_{34} := g(y_3)-g(y_4)$. Similarly, since $y_1 - y_3 = y_2 - y_4$, we also have that $v_{13} :=  g(y_1)-g(y_3)$ is parallel to $v_{24} := g(y_2)-g(y_4)$. 
Notice that 
\[\tfrac12 (v_{13} + v_{24}) = \tfrac12 (g(y_1)+g(y_2)) - \tfrac12 (g(y_3)+g(y_4)) \in \cH - \cH \subset \dir \cH,\]
which by parallelism of $v_{13}$ and $v_{24}$ can only be true if both $v_{13}, v_{24} \in \dir \cH$. 
In particular, $v_{13}$ and $v_{24}$ are perpendicular to $v_{12}$ and $v_{34}$.
This proves that $g(y_1),g(y_2),g(y_3),g(y_4)$ forms a rectangle in the same configuration than $y_1, y_2, y_3, y_4$. As a result, $g(\cR)$ is a hyperrectangle in the same configuration as $\cR$.
\end{proof}

As an immediate corollary, we get the following.
\begin{corollary} \label{cor:gdimb} $g(\cB)$ has affine dimension $p$.
\end{corollary}

\begin{lemma} \label{lem:glineb} Let $\cL$ be a line intersecting $\cB$. Then $g(\cL \cap \cB)$ is contained in a line. Furthermore, if $\cL'$ is another line intersecting $\cB$ parallel to $\cL$, then $g(\cL'\cap \cB)$ is parallel to $g(\cL\cap\cB)$.
\end{lemma}

\begin{proof} Let $x,x'$ and $x''$ three points of $\cL \cap \cB$. Then we can construct two hyperrectangles $\cR$ and $\cR'$ of $\cB$ with a common facet and such that $[xx']$ and $[x'x'']$ are two edges of $\cR$ and $\cR'$ orthogonal to that common facet. Since $g(\cR)$ and $g(\cR')$ are two hyperrectangles in the same configuration as $\cR$ and $\cR'$, theyx also share a common facet, and $[g(x)g(x')]$ and $[g(x')g(x'')]$ must be orthogonal to that common hyperfacet. They are thus parallel, so that $g(x), g(x')$ and $g(x'')$ are colinear.

For the second part of the proof, we can build a third hyperrectangle $\cR''$ of $\cB$ which contains two edges that are supported on $\cL$ and $\cL'$. Since $g(\cR'')$ is again a hyperrectangle, the images of these edges are parallel, and so must be $g(\cL\cap\cB)$ and $g(\cL'\cap\cB)$.
\end{proof}

\begin{lemma} \label{lem:addb} 
Assume that $0 \in \cB$ and that $g(0) = 0$.
It holds that $g(y_0+y_1) = g(y_0)+g(y_1)$ for all $y_0,y_1\in \cB$ such that $0 \notin (y_0y_1)$ and $y_0+y_1 \in \cB$.
\end{lemma}
\begin{proof} 
Take $y_0 \in \cB$ and $\cL_0 = \Vect(y_0)$. Then, by \lemref{glineb}, $g(\cL_0 \cap \cB) \subset \Vect(g(y_0))$. Now let $y_1 \in \cB$ such that $0 \notin (y_0y_1)$ and such that $y_0+y \in \cB$, and denote $\cL_1 = \Vect(y_1)$. Since $g$ conserves parallelism also by \lemref{glineb}, $g(\{y_1+\cL_0\} \cap \cB)$ is contained in a line with direction $\Vect(g(y_0))$, so that $g(\{y_1+\cL_0\}\cap \cB) \subset g(y_1) + \Vect(g(y_0))$.
Similarly, $g(\{y_0+\cL_1\}\cap \cB) \subset g(y_0) + \Vect(g(y_1))$. 
Now, since $y_0+y_1$ is the intersection of $y_0+\cL_1$ and $y_1+\cL_0$, there holds
\begin{align*} 
g(y_0+y_1) &\in g\(\{y_1+\cL_0\} \cap \cB\) ~ \bigcap~ g\(\{y_0+\cL_1\}\cap\cB\) \\
&\subset \{g(y_1) + \Vect(g(y_0))\}\ \bigcap\ \{g(y_0) + \Vect(g(y_1))\} = \{g(y_0) + g(y_1)\}. \qedhere
\end{align*} 
\end{proof}

\begin{lemma} \label{lem:multb} 
Assume that $0 \in \cB$ and that $g(0) = 0$.
It holds that $g(-y) = -g(y)$ for all $y\in\cB$.
\end{lemma}
\begin{proof} 
Let $y_0 \in \cB \setminus \{0\}$ and let $y_1 \in \cB \setminus \Vect(y_0)$. Then there exists (a small) $z$ in $\cB$ that is not in $\Vect(y_0)$ or $\Vect(y_1)$ and such that $y_0 \pm z$ and $y_1 \pm z$ are in $\cB$. Then, thanks to \lemref{addb}, 
\begin{align}
g(y_0)+g(-y_0) 
&= g(z+y_0)+g(z-y_0)-2g(z) \\
&= g(2z)-2g(z) \\
&= g(z+y_1)+g(z-y_1)-2g(z) \\
&= g(y_1) + g(-y_1),
\end{align}
so that $g(y_0)+g(-y_0) \in  \Vect g(y_0) \cap \Vect g(y_1)$. This last intersection is $\{0\}$ by injectivity of $g$.
\end{proof}

\begin{lemma} \label{lem:fgb} 
Assume that $0 \in \cB$ and that $g(0) = 0$.
It holds that $f(0) = 0$. 
\end{lemma}
\begin{proof} 
Let $x \in \cB \setminus \{0\}$. 
By \eqref{easy-ii}, $f(\hplane(x,-x) \cap \cB) \subset \hplane(g(x),-g(x))$, and by \lemref{multb}, $\hplane(g(x),-g(x)) = \Vect g(x)^\perp$. Thus, 
\[
f(0) \in \bigcap_{x\in \cB} \Vect g(x)^\perp = g(\cB)^\perp = \{0\},
\]
because $\dim g(\cB) = p$ thanks to Corollary \ref{cor:gdimb}. 
\end{proof}

\begin{proof}[Proof of \thmref{internal_point}]
Fix an arbitrary $x_0 \in \cB$ and define $g_0(x) := g(x+x_0)-g(x_0)$ and $f_0(x) := f(x+x_0)-g(x_0)$. Then $g_0$ and $f_0$ satisfy \eqref{internal_point_continuum}, $0 \in \cB - x_0$ and $g_0(0) = 0$. Therefore, thanks to \lemref{fgb}, $f_0(0) = 0$, and hence $f(x_0) = g(x_0)$. 
We have thus established that $f = g$ on $\cB$. 
\end{proof}

\begin{remark}
We initially thought that we could work in \thmref{internal_point} under more general conditions on $\cX$ and $\cY$. The result might hold, for example, if $\cX$ and $\cY$ are open and have a non-empty intersection. We do not know whether this is the case or not. 
%It is not even clear to us whether requiring that $\cX$ and $\cY$ intersect in that case is necessary or not. 
We note, however, that it is not sufficient that $\interior(\cX \cap \cY) \ne \emptyset$ --- a condition that would be in line with what was assumed in \thmref{external_point} and \thmref{mds}. Indeed, consider a situation where $\cX$ is the unit ball (open or closed) and $\cY = \cX \cup \cA$, $\cA := \{a u : a \ge 1\}$, with $u$ an arbitrary normed vector. In that case, $f = {\rm id}$ on $\cX$ and any $g = {\rm id}$ on $\cX$ and increasing along $\cA$ with $g(\cA) \subset \cA$, satisfy \eqref{internal_point_continuum}.
\end{remark}

\subsection{Discrete Asymptotic Setting}
\label{sec:internal point asymptotic}

We return to the discrete setting of \secref{internal point discrete} in an asymptotic setting where there is a pair of configurations satisfying \eqref{internal point problem} that become dense in a suitable manner. 

The following is inspired by \prpref{gBopen}. 

\begin{theorem}
\label{thm:internal point asymptotic}
Let $(x_n), (y_n), (x'_n), (y'_n)$ be bounded sequences of distinct points in $\bbR^d$ such that, for each $n$, $\{x_1, \dots, x_n; y_1, \dots, y_n\}$ and $\{x'_1, \dots, x'_n; y'_1, \dots, y'_n\}$ are equivalent configurations in the sense of \eqref{internal_equivalent_point_discrete}. 
For each $n$, let $f_n : \{x_i: i \ge 1\} \to \{x'_i : i \ge 1\}$ and $g_n : \{y_i: i \ge 1\} \to \{y'_i: i \ge 1\}$ be such that $f_n(x_i) = x'_i$ and $g_n(y_i) = y'_i$ for $i \in [n]$. 
Then $(f_n, g_n)$ is sequentially compact for the pointwise convergence topology.
Assume there is $N \subset \bbN$ such that  
\begin{gather}
%\begin{gathered}
\interior(\closure(x_n:n \in N)) \ne \emptyset, \label{internal asymptotic 1} \\
(x_n:n \in N) \subset \closure(y_n:n \in N), \label{internal asymptotic 2} \\
(x'_n:n \in N) \subset \interior(\closure(y'_n:n \in N)). \label{internal asymptotic 3} 
%\end{gathered}
\end{gather}
Then all the functions where $(f_n)$ accumulates are similarity transformations. 
If in addition either $(x_n)$ is dense in the convex hull of $(y_n)$, or $(y_n)$ is dense in its convex hull, then all the functions where $(f_n, g_n)$ accumulates are of the form $(L, L)$ where $L$ is a similarity transformation.
\end{theorem}

With more work, perhaps along the lines of what is done in \cite{klein,arias2017some}, it might be possible to derive a uniform rate of convergence.

\begin{proof}
As in the proof of \thmref{mds asymptotic}, the fact that $(f_n, g_n)$ is sequentially compact for the pointwise convergence topology is a standard result. 

To establish the second part of the statement, it suffices to consider the situation where $(f_n, g_n)$ converges pointwise on $\cX_\infty \times \cY_\infty := \{x_i : i \ge 1\} \times \{y_i : i \ge 1\}$. Let $(f, g)$ denote the limit, so that $\lim_n f_n(x_i) = f(x_i)$ and $\lim_n g_n(x_i) = g(x_i)$, for all $i \ge 1$. 
By the properties that define $f_n$ and $g_n$, we have $f(x_i) = x'_i$ and $g(y_i) = y'_i$ for $i \in \bbN$. In addition, we can write \eqref{internal_equivalent_point_discrete} as follows
\begin{equation}
\text{$\|x_i-y_k\| < \|x_i-y_l\| \iff \|f_n(x_i)-g_n(y_k)\| < \|f_n(x_i)-g_n(y_l)\|$, \quad for all $(i,k,l) \in [n] \times [n] \times [n]$,}
\end{equation}
and letting $n\to\infty$ this yields
\begin{equation}
\label{internal asymptotic proof 1}
\text{$\|x_i-y_k\| < \|x_i-y_l\| \implies \|f(x_i)-g(y_k)\| \le \|f(x_i)-g(y_l)\|$, \quad for all $(i,k,l) \in \bbN \times \bbN \times \bbN$.}
\end{equation}

%To simplify the notation, and without loss of generality, assume that $N = \bbN$ satisfies \eqref{intclcap}. 
%Letting $\cX := \interior(\closure\,\cX_\infty)$ and $\cY = \interior(\closure\,\cY_\infty)$, we have that $\cX$ and $\cY$ are open and have non-empty intersection, and 
Define $\cX_N := \{x_n:n \in N\}$ and $\cY_N := \{y_n:n \in N\}$.
Take $x \in \cX_N$. We claim that, if $\lim_{n \in N_1} y_n = x$ some $N_1 \subset N$, then $\lim_{n \in N_1} g(y_n) = f(x)$. Indeed, $f(x) \in f(\cX_N)$, and \eqref{internal asymptotic 2} also says that $g(\cY_\infty)$ is dense in $f(\cX_N)$ so that there is $N_2$ such that $\lim_{n \in N_2} g(y_n) = f(x)$ and $g(y_n), n \in N_2$ are all distinct. In particular, $y_n, n \in N_2$ are all distinct, so that we may assume that $\|x - y_n\| > 0$ for all $n \in N_2$.    
Now, for $n_2 \in N_2$, if $n_1 \in N_1$ is sufficiently large that $\|x-y_{n_1}\| < \|x-y_{n_2}\|$, we have $\|f(x)-g(y_{n_1})\| \le \|f(x)-g(y_{n_2})\|$ by \eqref{internal asymptotic proof 1}. And from this we deduce that 
\begin{equation}
\limsup_{n \in N_1} \|f(x)-g(y_{n_1})\| \le \limsup_{n_2 \in N_2} \|f(x)-g(y_{n_2})\| = 0.
\end{equation}

Given $x, x', x'' \in \cX_N$ such that $\|x -x'\| < \|x-x''\|$, let $N' \subset N$ and $N'' \subset N$ be such that $\lim_{n \in N'} y_n = x'$ and $\lim_{n \in N''} y_n = x''$. We now know that $\lim_{n \in N'} g(y_n) = f(x')$ and $\lim_{n \in N''} g(y_n) = f(x'')$. And when $n' \in N'$ and $n'' \in N''$ are large enough, $\|x -y_{n'}\| < \|x-y_{n''}\|$, so that $\|f(x)-g(y_{n'})\| \le \|f(x)-g(y_{n''})\|$ by \eqref{internal asymptotic proof 1}. Passing to the limit, we thus have $\|f(x)-f(x')\| \le \|f(x)-f(x'')\|$. We have thus established that $f$ is weakly isotonic on $\cX_N$ in the sense of \eqref{weakly_isotonic}. Because we also assume in \eqref{internal asymptotic 1} that $\interior(\closure\,\cX_N) \ne \emptyset$, we may apply \thmref{mds asymptotic} --- to the constant sequence of functions $(f)$ --- to deduce that $f$ coincides on $\cX_\infty$ with a similarity. Without loss of generality, we assume henceforth that $f(x) = x$ for all $\cX_\infty$. And we extend $f$ to $\bbR^d$ continuously.

Let $\cB = \interior(\closure\,\cX_N)$. We saw above that whenever $\lim_{n \in N_1} y_n = x \in \cX_N$ for some $N_1 \subset N$, then $\lim_{n \in N_1} g(y_n) = f(x)$. Now using the fact that $f(x) = x$, we see that $y_n \to x$ along some sequence within $\cY_N$ implies $g(y_n) \to x$. This clearly extends to $x \in \cB$, which then trivially yields $g(y) = y$ whenever $y \in \cB \cap \cY_N$ by considering the constant sequence $(y)$. We have thus established that $f = g = {\rm id}$ on the open set $\cB$ after extending $g$ by continuity.

Define $\cX := \closure\,\cX_\infty$. By assumption, either $\cX$ contains the convex hull of $\cY_\infty$, or $\cY := \closure\,\cY_\infty$ is convex. In the first case, we can directly apply \lemref{ball}. In the latter case, it is not clear how to do the same as we do not have an extension of $g$ to the entirety of $\cY$. But we can check that the same arguments underlying the proof of (the 2nd part of) \lemref{ball} apply verbatim.  \end{proof}

\section{Discussion}
\label{sec:discussion}

\paragraph{Error rates}
Beyond consistency results, one may want to derive error bounds. There are some precedents. 
For ordinal external unfolding (\secref{external}), \citet{massimino2021you} are able to obtain a bound for the very specific sampling model they consider in their work. 
For ordinal multidimensional scaling (\secref{mds}), an error bound or convergence rate is established in \cite{arias2017some}, although under slightly stronger assumptions. We note that the rate obtained for the one-dimensional case ($d=1$) is shown in \cite{ellenberg2019convergence} to be optimal. (We mention that the more general case of a metric space is considered in \cite{gouic2023recovering}, where a rate of convergence is also established.)
For ordinal internal unfolding (\secref{internal}), the literature seems totally devoid of any error rates. As we indicated earlier in the paper, despite its importance, there is a dearth of theory for this problem, and the consistency results derived here are the only such results that we know of besides the more loose, but nonetheless pioneering argumentation of \citet{shepard1966metric}. At the moment, we do not see a clear path towards establishing an error bound for this problem.

\paragraph{Missingness and noise}
We have avoided issues of missingness and noise in the data. 
While there is a good amount of theory available in the metric setting on these two issues, for example, in \cite{arias2020perturbation, anderson2010formal, arias2025stability, javanmard2013localization, li2020central, vishwanath2025minimax} (and of course the whole literature on {\em graph rigidity theory}) there is comparatively little in the ordinal setting.
A situation with some missingness, where only local comparisons are available, is considered in the context of ordinal embedding in \cite{arias2017some}, and some theory is developed in a framework that allows for erroneous comparisons in \cite{jain2016finite}.
However, we do not see an elegant and useful way of extending the continuous model to allow for missingness and/or noise in the comparisons.

\subsection*{Acknowledgements}
The work of EAC was partially supported by the US National Science Foundation (DMS 1916071).
The work of CB was supported by the Deutsche Foschungsgemeinschaft (German Research Foundation) on the French-German PRCI ANR ASCAI CA 1488/4-1 ``Aktive und Batch-Segmentierung, Clustering und Seriation: Grundlagen der KI''.
The work of DK was partially supported by the US National Science Foundation (NSF Medium Award CCF-2107547 and NSF Award CCF-1553288 CAREER).
%The authors have no competing interests to declare that are relevant to the content of this article.

{\small
\bibliographystyle{chicago}
\bibliography{../ref}
}

\end{document}

%% file: notation.tex
\def\sphere{\mathrm{S}}
\def\hplane{\mathrm{H}}
\def\interior{\mathrm{int}}
\def\closure{\mathrm{cl}}

\newtheorem{theorem}{Theorem}[section]
\newtheorem{proposition}[theorem]{Proposition}
\newtheorem{lemma}[theorem]{Lemma}
\newtheorem{corollary}[theorem]{Corollary}

\theoremstyle{remark}

\newtheorem{remark}[theorem]{Remark}

\numberwithin{equation}{section}
\numberwithin{figure}{section}

\def\beq{\begin{equation}} % \setcounter{equation}{1}}
\def\eeq{\end{equation}}
\def\beqn{\begin{eqnarray*}}
\def\eeqn{\end{eqnarray*}}
\def\Bitem{\begin{itemize}\setlength{\itemsep}{.2in}}
\def\bitem{\begin{itemize}\setlength{\itemsep}{.05in}}
\def\eitem{\end{itemize}}
\def\Benum{\begin{enumerate}\setlength{\itemsep}{.2in}}
\def\benum{\begin{enumerate}\setlength{\itemsep}{.05in}}
\def\eenum{\end{enumerate}}
\def\bmult{\begin{multline*}}
\def\emult{\end{multline*}}
\def\bcenter{\begin{center}}
\def\ecenter{\end{center}}
\def\bframe{\begin{frame}}
\def\eframe{\end{frame}}

\newcommand{\thmref}[1]{Theorem~\ref{thm:#1}}
\newcommand{\prpref}[1]{Proposition~\ref{prp:#1}}

\newcommand{\lemref}[1]{Lemma~\ref{lem:#1}}
\newcommand{\secref}[1]{Section~\ref{sec:#1}}

\def\cA{\mathcal{A}}
\def\cB{\mathcal{B}}

\def\cH{\mathcal{H}}

\def\cL{\mathcal{L}}

\def\cR{\mathcal{R}}
\def\cS{\mathcal{S}}

\def\cX{\mathcal{X}}
\def\cY{\mathcal{Y}}

\def\bbI{\mathbb{I}}

\def\bbN{\mathbb{N}}

\def\bbR{\mathbb{R}}
\def\bbS{\mathbb{S}}

\let\lac\{
\let\rac\}
\renewcommand{\{}{\left\lac}
\renewcommand{\}}{\right\rac}

\def\implies{\ \Rightarrow \ }
\def\iff{\ \Leftrightarrow \ }

\newcommand{\IND}[1]{\bbI\{ #1 \}}
\def\({\left(}
\def\){\right)}

\DeclareMathOperator{\Vect}{vect}
\DeclareMathOperator{\Span}{span}

\DeclareMathOperator{\dir}{dir}

\newcommand\ball{\mathrm{B}}